\documentclass{amsart}

\usepackage{amsmath,amsthm,amssymb,amsfonts}
\usepackage{hyperref}

\theoremstyle{plain}
\newtheorem{thm}{Theorem}%[section]
\newtheorem{prp}[thm]{Proposition}
\newtheorem{lem}[thm]{Lemma}

\theoremstyle{definition}
\newtheorem{dfn}[thm]{Definition}

\theoremstyle{remark}
\newtheorem{rmk}[thm]{Remark}

\newcommand{\Hom}{\mathrm{Hom}}

\newcommand{\Out}{\mathrm{Out}}

\newcommand{\id}{\mathrm{id}}

\newcommand{\CF}{\mathcal{F}}

\newcommand{\CH}{\mathcal{H}}

\newcommand{\CP}{\mathcal{P}}

\newcommand{\BQ}{\mathbb{Q}}

\newcommand{\BZ}{\mathbb{Z}}

\begin{document}
\title[]{Realizing fusion systems inside finite groups}
\author{Sejong Park}
\address{School of Mathematics, Statistics and Applied Mathematics, National University of Ireland, Galway, Ireland}
\email{sejong.park@nuigalway.ie}
\begin{abstract}
We show that every (not necessarily saturated) fusion system can be realized as a full subcategory of the fusion system of a finite group. This result extends our previous work \cite{Park2010} and complements the related result \cite{LearyStancu2007} by Leary and Stancu.
\end{abstract}
\maketitle

\section{Statements of the results}

Fix a prime $p$. Let $G$ be a finite group, and let $S$ be a $p$-subgroup of $G$.  We denote by $\CF_S(G)$ the category whose objects are the subgroups of $S$ and such that for $P, Q \leq S$ we have
\[
	\Hom_{\CF_S(G)}(P,Q) = \{ \varphi \colon P \to Q \mid \exists\, x \in G \text{ s.t.\@ } \varphi(u)=xux^{-1} \text{ for } u \in P \},
\]
where composition of morphisms is composition of functions.

The category $\CF_S(G)$ above is a \emph{fusion system} on $S$. If $S$ is a Sylow $p$-subgroup of $G$, then $\CF_S(G)$ is  \emph{saturated}, but not all saturated fusion systems are of this form. Those saturated fusion systems $\CF$ such that $\CF \neq \CF_S(G)$ for any finite group $G$ having $S$ as a Sylow $p$-subgroup are called \emph{exotic fusion systems}. We refer the reader to \cite{AKOBook} for precise definitions and a general introduction to the subject. In \cite{Park2010} we showed that every saturated fusion system $\CF$ on a finite $p$-group $S$ is of the form $\CF = \CF_S(G)$ for some finite group $G$ having $S$ as a subgroup. The point here is that we are not requiring that $S$ is a Sylow $p$-subgroup of $G$. In this short note, we observe that this result holds even when $\CF$ is not saturated.

\begin{thm} \label{T:main}
Let $\CF$ be a fusion system on a finite $p$-group $S$. Then there is a finite group $G$ having $S$ as a subgroup such that $\CF = \CF_S(G)$.
\end{thm}

Thus fusion systems are precisely those categories of the form $\CF_S(G)$ for some finite group $G$ and a $p$-subgroup $S$ of $G$. Leary and Stancu \cite{LearyStancu2007} showed that every fusion system is of the form $\CF_S(G)$ where $G$ is a (possibly infinite) group having $S$ as a Sylow $p$-subgroup, in the sense that every finite $p$-subgroup of $G$ is conjugate to a subgroup of $S$. Here $\CF_S(G)$ is defined exactly the same way as when $G$ is a finite group. Leary and Stancu's construction uses HNN extensions. As in \cite{Park2010}, the proof of Theorem \ref{T:main} uses a certain $S$-$S$-biset associated to the fusion system $\CF$, though we use a slightly different one here. We keep the notations of \cite{Park2010}.

\begin{dfn}
Let $\CF$ be a fusion system on a finite $p$-group $S$. A \emph{left semicharacteristic biset} for $\CF$ is a finite $S$-$S$-biset $X$ satisfying the following properties:
\begin{enumerate}
\item $X$ is \emph{$\CF$-generated}, i.e., every transitive subbiset of $X$ is of the form $S\times_{(Q,\varphi)} S$ for some $Q\leq S$ and some $\varphi\in\Hom_\CF(Q,S)$.
\item $X$ is \emph{left $\CF$-stable}, i.e., ${}_{Q}X \cong {}_{\varphi}X$ as $Q$-$S$-bisets for every $Q\leq S$ and every $\varphi\in\Hom_\CF(Q,S)$.
\end{enumerate}
\end{dfn}

A right semicharacteristic biset is defined analogously with right $\CF$-stability instead of left $\CF$-stability; a semicharacteristic biset is a biset which is both left and right semicharacteristic. When the fusion system is saturated, semicharacteristic bisets are parametrized by Gelvin and Reeh \cite{GelvinReeh2015} using a result of Reeh \cite{ReehMonoid}, and left semicharacteristic bisets can be parametrized analogously. A \emph{left characteristic biset} is a left semicharacteristic biset $X$ such that $|X|/|S| \not\equiv 0 \pmod p$. Broto--Levi--Oliver \cite[Proposition 5.5]{BLO2003} showed that every saturated fusion system has a left characteristic biset $X$. In \cite{Park2010}, we used this biset $X$ to construct the finite group $G$ in Theorem \ref{T:main} when $\CF$ is saturated. Here we show that every fusion system has a certain left semicharacteristic biset $X$ with an additional property which falls short of making $X$ left characteristic, but which still ensures that the proof in \cite{Park2010} carries over.

\begin{prp} \label{P:semichar}
Every fusion system $\CF$ on a finite $p$-group $S$ has a left semicharacteristic biset $X$ containing $S\times_{(S,\id)} S$.
\end{prp}

We are going to prove Proposition \ref{P:semichar} and Theorem \ref{T:main} in the next section.

\begin{rmk}
In \cite[Proposition 3.1]{Park2014}, a semicharacteristic biset containing $S\times_{(S,\id)} S$ is used for a saturated fusion system $\CF$ on a finite $p$-group $S$. Thus Proposition \ref{P:semichar} tells us that \cite[Proposition 3.1]{Park2014} holds for an arbitrary fusion system $\CF$.
\end{rmk}

\section{Semicharacteristic bisets for fusion systems}

Let $G$ be a finite group. A virtual $G$-set with rational coefficients is an element of the rational Burnside ring $\BQ\otimes_\BZ B(G)$, i.e., a formal sum
\[
	\sum_{H} c_H G/H
\]
where $H$ runs over a set of representatives of conjugacy classes of subgroups of $G$ and $c_H \in \BQ$. If the coefficients of a virtual $G$-set are all nonnegative integers, then it is simply a (isomorphism class of) finite $G$-set.

The key step of the proof of Proposition \ref{P:semichar} is the following lemma, which says roughly that every virtual $S$-set with rational coefficients can be stabilized (with respect to a given fusion system $\CF$) by adding a virtual $S$-set with nonnegative rational coefficients.

\begin{lem}[cf.\@ {\cite[Lemma 5.4]{BLO2003}}] \label{L:stabilizing S-set}
Let $\CF$ be a fusion system on a finite $p$-group $S$. Let $\CH$ be a collection of subgroups of $S$ which is closed under $\CF$-conjugation and taking subgroups. Let $X_0$ be a virtual $S$-set with rational coefficients such that $|X_0^P| = |X_0^{P'}|$ for all $P, P'\leq S$ with $P, P' \notin \CH$ which are $\CF$-conjugate. Then there is a virtual $S$-set $X$ with rational coefficients such that $|X^P| = |X^{P'}|$ for all $P, P'\leq S$ which are $\CF$-conjugate, $|X^P| = |X_0^P|$ for all $P\leq S$ with $P \notin \CH$, and $X - X_0$ is a virtual $S$-set with nonnegative rational coefficients.
\end{lem}

\begin{proof}
Consider an $\CF$-conjugacy class $\CP$ of subgroups of $S$ in $\CH$ which are maximal among such subgroups. Choose $P\in\CP$ such that $|X_0^P| \geq |X_0^{P'}|$ for all $P'\in\CP$. Set
\[
	X_1 = X_0 + \sum_{P'} \frac{|X_0^P| - |X_0^{P'}|}{|N_S(P')/P'|} S/P',
\]
where $P'$ runs over a set of representatives of the subgroups in $\CP$ up to $S$-conjugacy. Then for any $P'\in\CP$, we have $|X_1^{P'}| = |X_0^{P}| = |X_1^P|$. Note that $|X_1^P| = |X_0^P|$ for all $P \leq S$ with $P \notin \CH$, and hence $|X_1^P| = |X_1^{P'}|$ for all $P, P' \leq S$ with $P, P' \notin \CH \setminus \CP$ which are $\CF$-conjugate. Also, $X_1 - X_0$ is a virtual $S$-set with nonnegative rational coefficients. So by repeating this process we get a virtual $S$-set $X$ with the desired properties.
\end{proof}

Comparing the above lemma to \cite[Lemma 5.4]{BLO2003}, we see that here the lack of saturation is compensated for by allowing rational coefficients.	

\begin{proof}[Proof of Proposition \ref{P:semichar}]
Let
\[
	Y_0 = \sum_{\alpha \in \Out_\CF(S)} S\times_{(S,\alpha)} S.
\]
Then $Y_0$ satisfies the assumption of Lemma \ref{L:stabilizing S-set} with respect to the product fusion system $\CF\times\CF_S(S)$ on $S\times S$ and $\CH = \{ \Delta(P,\varphi) \mid P < S, \varphi\in\Hom_\CF(P, S) \}$. (See \cite[Definition I.6.5, Theorem I.6.6]{AKOBook} for the definition and properties of the product fusion system.) Thus Lemma \ref{L:stabilizing S-set} implies that there is a virtual $S$-set $Y$ with nonnegative rational coefficients which is $\CF$-generated and left $\CF$-stable and which contains $S\times_{(S,\id)} S$. Let $m$ be a large enough positive integer such that $X=mY$ is a (finite) $S$-set (with nonnegative integer coefficients). Then $X$ is a left semicharacteristic biset for $\CF$ containing $S\times_{(S,\id)} S$. 
\end{proof}

\begin{proof}[Proof of Theorem \ref{T:main}]
Let $\CF$ be a fusion system on a finite $p$-group $S$ and let $X$ be a left semicharacteristic biset for $\CF$ containing $S\times_{(S,\id)} S$. Let $G$ be the group of automorphisms of $X$ viewed as a right $S$-set, i.e., the group of bijections $f\colon X\to X$ such that $f(xs) = f(x)s$ for all $x\in X$ and $s\in S$. Then $S$ embeds into $G$ via
\[
	S\to G,\quad s\mapsto (x \mapsto sx).
\]
The proof of \cite[Theorem 6]{Park2010} applies verbatim to this situation. Thus we have $\CF = \CF_S(G)$.
\end{proof}

\end{document}